\pdfoutput=1
\documentclass[12pt,a4paper]{article}
\usepackage[T1]{fontenc}
\usepackage{amsthm,amssymb,amsmath}
\usepackage{stmaryrd}
\usepackage{float} 
\usepackage{enumerate}
\usepackage{array} 
\usepackage{setspace}
\usepackage{cite}
\usepackage{varioref}
\usepackage{hyperref}
\newtheorem{thm}{Theorem}[section] 
\newtheorem{dfn}{Definition}[section] 
\newtheorem{lem}{Lemma}[section] 
\newtheorem{prop}{Proposition}[section]
\newtheorem{cor}{Corollary}[section]
\newtheorem{rem}{Remark}[section]
\newtheorem{claim}{Claim}[section]
\author{Mrunal Hardikar, Anuradha Garge}
\title{Integral quadratic forms over a ring of $p$-adic integers}

\begin{document}
	\maketitle
	Lee \cite{JunginLee} proved a necessary and sufficient condition that an integral quadratic form $\sum_{i=1}^{m} a_iX_i^2$ is universal over $M_2(\mathbb{Z})$. For a positive integer $n \geq 2$, Lee defined $f(n)$ to be the smallest positive integer $m$ such that for every pairwise coprime $a_1, a_2, \ldots a_m \in \mathbb{Z}$, $\sum_{i=1}^{m}a_iX_i^2$ is universal over $M_n(\mathbb{Z})$. He gave bounds on $f(n)$ in \cite{JunginLee}. Koo and Lee further improved the bounds in \cite{Koo2025}.
	
	This paper is organized as follows. Let $\mathbb{Z}_p$ be the ring of $p$-adic integers. In Section \ref{Sec M2Zp Main theorems}, we give necessary and sufficient condition for a quadratic form $\sum_{i=1}^{m}a_iX_i^2$ to be universal over $M_2(\mathbb{Z}_p)$ for $p=2$ and for an odd prime $p$. Consequently we express matrices of over $\mathbb{Z}_p$ as sum of squares. For a positive integer $n \geq 3$, we define $f(n)$ to be the smallest positive integer $m$ such that for every $p$-adic integers $a_1,a_2, \ldots a_m$ with at least three of them units, the diagonal quadratic form $\sum_{i=1}^{m}a_iX_i^2$ is universal over $M_n(\mathbb{Z}_p)$. In Section \ref{Sec bounds f(n)},we find bounds on $f(n)$ for $n \geq 3$. 
	
	\section{Preliminaries}\label{Sec Preliminaries}
	For every $n \geq 1$, let $A_n= \mathbb{Z}/p^n\mathbb{Z}$. It is the ring of classes of integers (mod $p^n$). We have a homomorphism $$\phi_n : A_n \rightarrow A_{n-1}$$ defined by $\phi(x+ p^n \mathbb{Z}) = x+p^{n-1} \mathbb{Z}$. Clearly, $\phi_n$ is a surjective ring homomorphism.
	\begin{eqnarray*}
		\mbox{ker }\phi_n &=& \{x +p^n \mathbb{Z} \mbox{ } | \mbox{ }\phi_n(x) = 0+p^{n-1} \mathbb{Z}\} \\
		&=& \{x+p^n \mathbb{Z} \mbox{ }| \mbox{ }x + p^{n-1} \mathbb{Z}=0+p^{n-1} \mathbb{Z}\} \\
		&=& \{x+p^n \mathbb{Z} \mbox{ } | \mbox{ }x= p^{n-1}K \mbox{ for some }k \in \mathbb{Z}\} \\
		&=& p^{n-1} (\mathbb{Z}/p^n\mathbb{Z}) \\
		&=& p^{n-1}A_n
	\end{eqnarray*}
	Hence, the sequence $$\cdots \rightarrow A_n \rightarrow A_{n-1} \rightarrow \cdots \rightarrow A_2 \rightarrow A_1$$ forms a projective system indexed by integers $\geq 1$.
	\begin{dfn}\label{dfn padic integer as a limit of projective system}
		The ring of $p$-adic integers $\mathbb{Z}_p$ is the projective limit of the system $(A_n, \phi_n)$ defined above.
	\end{dfn}
	By definition, an element of $\mathbb{Z}_p$ is a sequence $x=(\cdots,  x_n, \cdots , x_1)$ with $x_n \in A_n$ and $\phi_n(x_n)= x_{n-1}$ if $n \geq 2$. Define coordinate-wise addition and multiplication on $\mathbb{Z}_p$. With these operations, $\mathbb{Z}_p$ forms a commutative ring with unity. In other words, $\mathbb{Z}_p$ is a subring of the product $\prod_{n \geq 1} A_n$.
	Let $\epsilon_n : \mathbb{Z}_p \rightarrow A_n$ be the function which associates to a $p$-adic integer $x$ its $n$-th component $x_n$. We state the following propositions from \cite{JPSerre}.
	\begin{prop}\label{prop exact seq of abelian group}
		The sequence $0 \longrightarrow \mathbb{Z}_p \overset{p^n}{\longrightarrow} \mathbb{Z}_p  \overset{\epsilon_n}{\longrightarrow} A_n \longrightarrow 0 $ is an exact sequence of abelian groups.
	\end{prop}
	Hence, we can identify $\mathbb{Z}_p/p^n\mathbb{Z}_p$ with $A_n= \mathbb{Z}/p^n\mathbb{Z}$.
	\begin{prop}\label{prop units in zp neces and suff}
		\begin{enumerate}[(a)]
			\item For an element of $\mathbb{Z}_p$ to be invertible it is necessary and sufficient that it is not divisible by $p$.
			\item If $U$ denotes group of invertible elements of $\mathbb{Z}_p$, every nonzero element of $\mathbb{Z}_p$ can be uniquely written in the form $p^nu$ with $u \in U$ and $n \geq 0$. 
		\end{enumerate}
	\end{prop}
	\begin{dfn}\label{dfn padic valuation}
		Any $x \in \mathbb{Z}_p$, can be uniquely expressed as $x=p^n \cdot u$ for some integer $n \geq 0$ and for some unit $u$.The integer $n$ is called $p$-adic valuation of $x$ and is denoted by $v_p(x)$. Define $v_p(0) = \infty$.
	\end{dfn}
	The group of units in $\mathbb{Z}_p$ consists of all $p$-adic integers $x$ with $v_p(x) = 0$. We have, $v_p(xy) = v_p(x)+v_p(y)$ and $v_p(x+y) \geq \mbox{inf }\{v_p(x),v_p(y)\}$. From these formulae, it follows that $\mathbb{Z}_p$ is an integral domain. The principal ideals of $\mathbb{Z}_p$ have an intersection $\{0\}$:
	\begin{equation*}
		\mathbb{Z}_p \supset p\mathbb{Z}_p \supset \cdots \supset p^k\mathbb{Z}_p \supset \cdots \bigcap_{k \geq 0} p^k\mathbb{Z}_p=\{0\}
	\end{equation*}
	In fact, for any non-zero $x \in \mathbb{Z}_p$ with $v_p(x)=k$, $a \notin (p^{k+1})$. These principal ideals are the only non-zero ideals of the ring of $p$-adic integers. Hence, $\mathbb{Z}_p$ is a principal ideal domain. The ring $\mathbb{Z}_p$ has a unique maximal ideal, namely $p\mathbb{Z}_p = \mathbb{Z}_p - U$. With the canonical representation $x=p^n \cdot u$ for every non-zero $x$, $\mathbb{Z}_p$ is unique factorization domain.
	\section{Universality of $\sum_{i=1}^{m}a_iX_i^2$ over $M_2(\mathbb{Z}_2$)} \label{Sec M2Zp Main theorems}
	Let $M_2(\mathbb{Z}_p)$ be the ring of matrices of order $2 \times 2$, over $p$-adic integers. For a quadratic form $\sum_{m=1}^{m}a_iX_i^2$ to be universal over a ring of integral matrices of order $2 \times 2$, one of the necessary and sufficient conditions in \cite{JunginLee}, needed the coefficients $a_1, a_2, \ldots a_m$ to be coprime. In case of $\mathbb{Z}_p$, the necessary and sufficient condition reduces to the following.
	\begin{thm}\label{thm Main theorem p=2}
		A quadratic form $\sum_{i=1}^{m} a_iX_i^2$ is universal over $M_2(\mathbb{Z}_2)$ if and only if at least two number from $a_1,a_2, \ldots a_n$ are units.
	\end{thm}
	We prove the following Lemma.
	\begin{lem} \label{aX2 over Z2}
		A quadratic form $aX^2$, where $a$ is unit in $ \mathbb{Z}_2 $ is not universal over $M_2(\mathbb{Z}_2)$. In fact, for any ideal $2^j\mathbb{Z}_2$ of $\mathbb{Z}_2$, $aX^2$ is not universal over $M_2(\mathbb{Z}_2/2^j\mathbb{Z}_2)$.
	\end{lem}
	\begin{proof}
		This can be proved similarly as Lemma 15 of \cite{Nullwala2022}.
	\end{proof}
	\begin{lem}\label{a1x1^2 + a_2^jX_2^2 over Z2}
		A quadratic form $aX_1^2 + 2^jX_2^2 $ for unit $a \in \mathbb{Z}_2$, and $j \geq 1$ is not universal over $M_2(\mathbb{Z}_2)$.
	\end{lem}
	\begin{proof}
		Assume contrary that the quadratic form $aX_1^2 + 2^jX_2^2 $ for unit $a \in \mathbb{Z}_2$ and for $j \geq 1$ is universal over $M_2(\mathbb{Z}_2)$. Hence the quadratic form $X_1^2$ is universal over $M_2(\mathbb{Z}/2\mathbb{Z}_2)$, a contradiction to Lemma \ref{aX2 over Z2}.
	\end{proof}
	Note that from Lemma \ref{aX2 over Z2} and Lemma \ref{a1x1^2 + a_2^jX_2^2 over Z2} for a quadratic form $\sum_{i=1}^{m} a_iX_i^2$ to be universal over $M_2(\mathbb{Z}_2)$, $m \geq 2$.
	We now give the proof of the necessary condition of Theorem \ref{thm Main theorem p=2}.
	\begin{prop}\label{necce condition main thm p=2}
		If a quadratic form $\sum_{i=1}^{m} a_iX_i^2$ is universal over $M_2(\mathbb{Z}_2)$, then at least two number from $a_1,a_2, \ldots a_n$ are units.
	\end{prop}
	\begin{proof}
		It is clear that at least one of the numbers from $a_1,a_2, \ldots a_n$ is unit. By Lemma \ref{a1x1^2 + a_2^jX_2^2 over Z2}, $2$ does not divide $m-1$ of the numbers $a_1,a_2, \ldots a_n$. Hence at least two of the numbers from $a_1,a_2, \ldots a_n$ are units.
	\end{proof}
	\begin{thm} \label{suff condition main thm p=2}
		If at least two of the numbers from $a_1,a_2, \ldots a_n$ are units, then the quadratic form $\sum_{i=1}^{m} a_iX_i^2$ is universal over $M_2(\mathbb{Z}_2)$.
	\end{thm}
	\begin{proof}
		We begin as the Theorem $2.1$ of \cite{JunginLee}.  Let $\sum_{i=1}^{m} a_iX_i^2$ be a quadratic form over $M_2(\mathbb{Z}_2)$ such that $a_1$ and $a_2$ are units. To prove universality of such quadratic form over $M_2(\mathbb{Z}_2)$, let $ \begin{bmatrix}
			A & B  \\
			C & D  
		\end{bmatrix} $ $\in \mathbb{M}_2(\mathbb{Z}_2)$, be arbitrary. 
		\begin{equation} 
			\sum_{i=1}^{m-1} a_i \begin{bmatrix}
				x_i & y_i  \\
				z_i & c_i-x_i  
			\end{bmatrix}^2 +a_m \begin{bmatrix}
				0 & N\\
				1 & 0
			\end{bmatrix}^2 = \begin{bmatrix}
				A & B \\
				C & D
			\end{bmatrix} \label{eq 1}
		\end{equation}
		if and only if the following four equations hold.
		\begin{subequations}
			\begin{align}
				\sum_{i=1}^{m-1} a_i (x_i^2 + y_iz_i) + Na_m &= A  \label{eq 2a}\\
				\sum_{i=1}^{m-1}a_ic_iy_i &= B  \label{eq 2b}\\
				\sum_{i=1}^{m-1} a_ic_iz_i &= C \label{eq 2c} \\
				2 \sum_{i=1}^{m-1} a_ic_ix_i &= {A-D+ \sum_{i=1}^{m-1} a_ic_i^2} \label{eq 2d}
			\end{align}
		\end{subequations}
		In order to prove universality of the quadratic form, we need to prove the existence of $x_i$'s, $y_i$'s,$z_i$'s, and $N$ in $\mathbb{Z}_2$ satisfying above four equations. Once we have $x_i$'s, $y_i$'s, $z_i$'s satisfying Equations \eqref{eq 2b}, \eqref{eq 2c} and \eqref{eq 2d}   and if $a_m$ is unit then we can find $N$ satisfying Equation \eqref{eq 2a}. Suppose $a_m$ is not unit. Hence $a_m=2^j u$ for $j \geq 1$ and unit $u$ in $\mathbb{Z}_2$. Choose $c_1=a_1$ and 
		\begin{equation*}
			c_2 = \begin{cases}
				a_2 &\mbox{ if } (A-D) \equiv 0 ( \mbox{mod } 2)\\
				2a_2 &\mbox{ if } (A-D) \equiv 1 ( \mbox{mod } 2)
			\end{cases}
		\end{equation*}	
		\begin{claim}\label{clm gcd aici =1}
			GCD $(a_1c_1, a_2c_2, \ldots a_{m-1}c_{m-1}) = 1$.\\ Since $a_1$ and $c_1$ are units, $a_1c_1$ is unit and hence the Claim \ref{clm gcd aici =1}.
		\end{claim}
		\begin{claim}\label{clm 2aicizi for p=2}
			\begin{equation*}
				A-D + \sum_{i=1}^{m-1} a_ic_i^2 \equiv 0 \mbox{ } (\mbox{mod }2).
			\end{equation*}
		\end{claim}
		We prove this claim as follows:
		\begin{eqnarray*}
			\sum_{i=1}^{m-1} a_ic_i^2 &=& a_1c_1^2+a_2c_2^2 + \sum_{i=3}^{m-1} a_ic_i^2 \\
			&\equiv & a_1^3+a_2c_2^2 \mbox{ } (\mbox{mod } 2) \mbox{... by choice of } c_i's
		\end{eqnarray*}
		Since $a_1$ is units, $a_1^3 \equiv 1 (\mbox{mod }2)$. Since $c_2 \equiv A-D+1 \mbox{ } (\mbox{mod } 2)$, $c_2^2 \equiv A-D+1 \mbox{ } (\mbox{mod } 2)$. Since $a_2$ is unit, $a_2c_2^2 \equiv A-D+1 \mbox{ } (\mbox{mod } 2)$. Hence
		\begin{eqnarray*}
			A-D + \sum_{i=1}^{m-1}a_ic_i^2 &\equiv& A-D + (1+(A-D)+1) (\mbox{mod }2)\\
			&\equiv & 0 \mbox{ }(\mbox{mod }2)
		\end{eqnarray*}
		Hence, the Claim \ref{clm 2aicizi for p=2}\\
		By Claim \ref{clm gcd aici =1}, we can find $y_i$'s and $z_i$'s satisfying Equations \eqref{eq 2b} and \eqref{eq 2c}. By Claim \ref{clm 2aicizi for p=2}, we can find $x_i$'s satisfying Equation \eqref{eq 2d}. Now we need to find $N$ satisfying Equation \eqref{eq 2a}. Let $a_m=2^j u$ for some unit $u$ in $\mathbb{Z}_2$ and for an integer $j \geq 1$. We use translational replacements as in \cite{JunginLee}. Let
		\begin{equation*}
			E= \sum_{i=1}^{m-1} a_i (x_i^2 + y_iz_i) - A
		\end{equation*}
		We will prove that $E$ is multiple of $a_m$. Since $2^j$ is the highest power of $2$ that divides $a_m$, it is sufficient to prove that $2^j$ divides $E$. With the following replacements,
		\begin{eqnarray*}
			y_1 &\mapsto& y_1+ku\\
			z_1 &\mapsto& z_1+tu
		\end{eqnarray*}
		for some $t$ and $k$ in $\mathbb{Z}_2$, Equations \eqref{eq 2b}, \eqref{eq 2c} and \eqref{eq 2d} still hold and $E$ changes to $E+a_1u(y_1t+k(z_1+tu))$. Choose $t \in \mathbb{Z}_2$ such that 
		\begin{equation*}
			t= \begin{cases}
				0 &\mbox{ if } z_1 \mbox{ is unit.}\\
				1 &\mbox{ if } z_1 \mbox{ is not unit.}
			\end{cases}
		\end{equation*}
		Since $u$ is unit and by choice of $t$, $(z_1+tu)$ is unit. Choose $$k=(z_1+tu)^{-1}[(2^j-E)(ua_1)^{-1} -ty_1].$$
		With this choice of $k$ and $t$, $2^j = \mbox{New }E$. Hence $E$ is multiple of $a_m$. We can find $N \in \mathbb{Z}_2$ satisfying Equation \eqref{eq 2a}. Therefore, the given quadratic form is universal.
	\end{proof}
	Now we prove the following corollaries of the above theorem.
	\begin{cor} \label{cor f(2) for p=2}
		Any matrix $ \begin{bmatrix}
			A & B  \\
			C & D  
		\end{bmatrix} $ in $M_2(\mathbb{Z}_2)$, can be expressed as $a_1 X_1^2 + a_2 X_2^2$ where $a_1$ and $a_2$ are units.
	\end{cor}
	\begin{proof}
		We consider the following two cases: \\
		\textbf{Case (i): } $A-D \equiv 1 (\mbox{mod }2)$. We can prove the existence of $x_1,y_1,z_1$ and $N$ $\in \mathbb{Z}_2$ (as in the proof of Theorem \ref{thm Main theorem p=2}) satisfying the equation
		\begin{equation*} \label{eq 3}
			a_1 \begin{bmatrix}
				x_1 & y_1  \\
				z_1 & 1-x_1  
			\end{bmatrix}^2  +a_2\begin{bmatrix}
				0 & N \\
				1 & 0
			\end{bmatrix}^2 = \begin{bmatrix}
				A & B \\
				C & D
			\end{bmatrix}
		\end{equation*} 
		\textbf{Case (ii): }$A-D \equiv0 (\mbox{mod }2)$. We can prove the existence of $x_1,x_2,y_1,y_2,z_1$ and $z_2$ $\in \mathbb{Z}_2$ (as in the proof of Theorem \ref{thm Main theorem p=2}) satisfying the equation
		\begin{equation*} \label{eq 4}
			a_1 \begin{bmatrix}
				x_1 & y_1  \\
				z_1 & 1-x_1  
			\end{bmatrix}^2  + a_2 \begin{bmatrix}
				x_2 & y_2  \\
				z_2 & 1-x_2
				
			\end{bmatrix}^2 = \begin{bmatrix}
				A & B \\
				C & D
			\end{bmatrix}
		\end{equation*}
	\end{proof}
	\begin{rem}
		Note that, if $(A-D) \equiv 1 (\mbox{mod }2)$, then we do not need $a_2$ to be unit.  We can find $N$ satisfying Equation \eqref{eq 3} by using transnational replacement as in the proof of Theorem \ref{thm Main theorem p=2}.
	\end{rem}
	\begin{cor} \label{Cor matrix as a sum of squares p=2}
		Any matrix $\begin{bmatrix}
			A & B \\
			C & D
		\end{bmatrix}$ in $M_2(\mathbb{Z}_2)$ is a square or a sum of two squares.
	\end{cor}
	\begin{proof}
		If $\begin{bmatrix}
			A & B \\
			C & D
		\end{bmatrix}$ is not a square, then from Corollary \ref{cor f(2) for p=2}, we can find $X_1$ and $X_2$ such that 
		\begin{equation*}
			X_1^2 + X_2^2= \begin{bmatrix}
				A & B\\
				C& D
			\end{bmatrix}
		\end{equation*}
	\end{proof}
	Now we prove the necessary and sufficient condition for an odd prime $p$.
	\begin{thm} \label{necce and suff condition main thm odd p}
		A quadratic form $\sum_{i=1}^{m} a_iX_i^2$ is universal over $M_2(\mathbb{Z}_p)$ if and only if at least two numbers from $a_1,a_2, \ldots a_n$ are units.
	\end{thm}
	We prove the necessary and sufficient conditions as follows.
	\begin{prop}\label{necce condition main thm odd p}
		If a quadratic form $\sum_{i=1}^{m} a_iX_i^2$ is universal over $M_2(\mathbb{Z}_p)$, then at least two numbers from $a_1,a_2, \ldots a_n$ are units.
	\end{prop}
	\begin{proof}
		 If a quadratic form $\sum_{i=1}^{m} a_iX_i^2$ is universal over $M_2(\mathbb{Z}_p)$, then it is universal over $M_2(\mathbb{Z}_p/p\mathbb{Z}_p)$. Lemma \ref{aX2 over Z2} and Lemma \ref{a1x1^2 + a_2^jX_2^2 over Z2} hold for an odd prime $p$ also. Hence at least two number from $a_1,a_2, \ldots a_n$ is unit.
	\end{proof}
	\begin{thm} \label{Zp sufficient case}
		If at least two of the numbers from $a_1,a_2, \ldots a_m$ are units, then the quadratic form $\sum_{i=1}^{m} a_iX_i^2$ is universal over $M_2(\mathbb{Z}_p)$.
	\end{thm}
	\begin{proof}
		Let $a_1$ and $a_2$ be units. We can proceed with the four equations as above. Choose $c_i$'s as in Theorem \ref{thm Main theorem p=2}, so that Claims 2.1 and 2.2 hold.  We can choose $x_i$'s, $y_i$'s, $z_i$'s satisfying Equations \eqref{eq 2b}, \eqref{eq 2c} and \eqref{eq 2d}. If $a_m$ is unit, we can choose $N$ such that Equation \eqref{eq 2a} holds. If $a_m$ is not unit, let $a_m = p^ju$ for some unit $u \in \mathbb{Z}_p$. We can use translational replacements as in earlier theorem. Let $z_1=p^n \alpha$ for some unit $\alpha \in \mathbb{Z}_p$. Choose $t \in \mathbb{Z}_p$ such that
		\begin{equation*}
			t= \begin{cases}
				0 &\mbox{ if } n=0 \\
				u^{-1} &\mbox{ if } n \neq 0
			\end{cases}
		\end{equation*}
		If $n=0$, $z_1+tu = \alpha$. Since $\alpha$ is unit, $z_1 + tu$ is unit. If $n \neq 0$, $z_1+tu = p^n\alpha +1 \equiv 1 (\mbox{mod }p)$. Hence, $z_1+tu$ is unit.
		As in earlier theorem, we can find $N$ satisfying Equation \eqref{eq 2a}. In this case, new $E$ will be multiple of $p^j$. Hence the proof.  
	\end{proof}
	We state the following corollaries. The proof is similar to $\mathbb{Z}_2$ case. 
	\begin{cor}\label{Cor f(2) for odd p}
		Let $a_1$ and $a_2$ be units. Then any matrix $ \begin{bmatrix}
			A & B  \\
			C & D  
		\end{bmatrix} $ in $M_2(\mathbb{Z}_p)$, can be expressed as $a_1X_1^2 + a_2X_2^2$.
	\end{cor}
	\begin{cor} \label{Cor matrix as a  square odd p}
		Any matrix $\begin{bmatrix}
			A & B \\
			C & D
		\end{bmatrix}$ in $M_2(\mathbb{Z}_p)$ is a square or a sum of two squares.
	\end{cor}
	\section{Bounds on $f(n)$} \label{Sec bounds f(n)}
	For a positive integer $n \geq 3$, define $f(n)$ to be the smallest positive integer $m$ such that for every $p$-adic integers $a_1,a_2, \ldots a_m$, with at least three of them units,  $\sum_{i=1}^{m}a_iX_i^2$ is universal over $M_n(\mathbb{Z}_p)$.
	Now we consider matrices of order $3$ over $\mathbb{Z}_p$. We proceed as Theorem 2.2 of \cite{Koo2025}.
	\begin{thm} \label{thm f(3) for p=2}
		For any prime $p$, $f(3) \leq 4$.
	\end{thm} 
	\begin{proof}
		
			Let $A = \begin{bmatrix}
				p_1 & p_2 & p_3 \\
				p_4 & p_5 & p_6 \\
				p_7 & p_8 & p_9 \\
			\end{bmatrix} \in M_3(\mathbb{Z}_p)$ be arbitrary. Suppose $a_1$ $a_2$ and $a_3$ are units. We need to prove the existence of the matrices $X_1$,$X_2$ and $X_3$ and $X_4$  such that $A= a_1X_1^2 + a_2X_2^2 + a_3X_3^2 + a_4X_4^2$. Let $X_i = \begin{bmatrix}
				0 & 0 \\
				0 & Y_i
			\end{bmatrix}$ ($Y_i \in M_2(\mathbb{Z}_2)$, $i=1,2$), $X_3 = \begin{bmatrix}
				0 & 0 & 1 \\
				1 & x_3 & 0 \\
				y_3 & z_3 & w_3
			\end{bmatrix}$ and $X_4= \begin{bmatrix}
				0 & 1 & 1\\
				1 & x_4 & 0 \\
				y_4 & z_4 & w_4
			\end{bmatrix}$. Then $A= a_1X_1^2 + a_2X_2^2 + a_3X_3^2 + a_4X_4^2$ holds if and only if the following equations hold.
			\begin{subequations}
				\begin{align}
					a_3y_3 + a_4 (1+y_4) &= p_1  \label{eq 3a} \\ 
					a_3z_3 + a_4z_4 + a_4x_4 &= p_2 \label{eq 3b} \\
					a_3w_3 + a_4w_4 &= p_3 \label{eq 3c} \\
					a_3x_3 + a_4x_4 &= p_4 \label{eq 3d} \\
					a_3z_3 + a_3w_3y_3 + a_4z_4 + a_4y_4w_4 &= p_7 \label{eq 3e} \\
					a_1Y_1^2 + a_2Y_2^2 + g(X_3, X_4) &= \begin{bmatrix}
						p_5 & p_6 \\
						p_8 & p_9
					\end{bmatrix} \label{eq 3f}
				\end{align}
			\end{subequations}
			where $g(X_3, X_4)$ is a function of $X_3$ and $X_4$ and not a function of $Y_1, Y_2$. We can find $x_i, y_i, z_i, w_i$ and $Y_i$ which satisfy above equations as follows:
			\begin{enumerate}
				\item[(a)] Since $a_3$ and $a_4$ are coprime, we can choose $y_3$, $y_4$, $w_3$, $w_4$ satisfying Equations \eqref{eq 3a} and \eqref{eq 3c}.
				\item[(b)] Substitute these values in Equation \eqref{eq 3e} and find $z_3$ and $z_4$ satisfying Equation \eqref{eq 3e}.
				\item[(c)] Substitute values of $z_3$ and $z_4$ in Equation \eqref{eq 3b}, so that we have
				\begin{equation}
					a_4x_4 = p_2-(a_3z_3+a_4z_4)  \label{eq 3g}
				\end{equation}
				\item[(d)] Now add Equations \eqref{eq 3g} and \eqref{eq 3d} to get $a_3x_3 + 2 a_4x_4 = p_2-(a_3z_3+a_4z_4) + p_4$. Since $a_3$ and $2a_4$ are coprime, we can find values of $x_3$ and $x_4$ satisfying Equations \eqref{eq 3b} and \eqref{eq 3d}.
				\item[(e)] Since $a_1$ and $a_2$ are units, from Corollary \ref{cor f(2) for p=2}, and Corollary \ref{Cor f(2) for odd p}, we can find $Y_1$ and $Y_2$ satisfying Equation \eqref{eq 3f}.
			\end{enumerate}
	\end{proof}
	\begin{cor}
		Any matrix of order $3$ over $Z_p$ can be expressed as a sum of at most $4$ squares.
	\end{cor}
	\begin{proof}
		From Theorem \ref{thm f(3) for p=2}, the quadratic form $\sum_{i=1}^{4} X_i^2$ must be universal. Similar calculations give the following values: (One of the many solutions) For any arbitrary $z_3 \in \mathbb{Z}_p$,
		\begin{eqnarray*}
			y_3&=& 0 \\
			y_4 &=& p_1-1 \\
			w_3&=& 0 \\
			w_4 &=& p_3 \\
			x_4 &=& p_4-x_3 \\
			x_3 &=& p_7-p_2+p_4-(p_1-1)p_3 \\
			z_4 &=& p_2-x_4-z_3
		\end{eqnarray*}
	\end{proof}
	
	Now we assume $n \geq 4$. We proceed exactly as~\cite{Koo2025}. Use the same notations. Let $E_{i,j} \in M_n(\mathbb{Z}_p)$ be the matrix unit whose $(i,j)$-th entry is 1 and other entries are 0. Let
	\begin{equation*}
		J=E_{n,n-1}+E_{n-1,n-2}+\ldots+E_{2,1}=\begin{bmatrix}
			0&\cdots &0&0&0\\
			1&\cdots &0&0&0\\
			\vdots&\ddots&\vdots&\vdots&\vdots \\
			0&\ldots & 1&0&0 \\
			0&\ldots&0&1&0
		\end{bmatrix}
	\end{equation*}
	and
	\begin{equation*}
		J_t=t\sum_{i=1}^{\left\lfloor \dfrac{n}{2}\right\rfloor}E_{n-2i+2,n-2i+1} + \sum_{i=1}^{\left\lfloor \dfrac{n-1}{2}\right\rfloor}E_{n-2i+1,n-2i}=\begin{bmatrix}
			0&\ldots&0&0&0&0 \\
			t'&\ldots&0&0&0&0 \\
			\vdots&\ddots&\vdots&\vdots&\vdots \\
			0&\ldots&t&0&0&0\\
			0&\ldots&0&1&0&0 \\
			0&\ldots&0&0&t&0
		\end{bmatrix}
	\end{equation*}
	for $t \in \mathbb{Z}_p$, where $t'=t$ if $n$ is even and $t'=1$ if $n$ is odd. Simple matrix calculations show that $J^2=E_{n,n-2} + E_{n-1,n-3} + \cdots +E_{3,1}$ and $J_t^2 = tJ^2$. we state the Lemmas 2.3 and 2.4 from ~\cite{Koo2025}.
	\begin{lem} \label{Lemma 2.3 of Koo}
		Assume $A \in M_n(\mathbb{Z}_p)$ satisfies $A_{i,i-2}=1$ for $3 \leq i \leq n$ and $A_{i,j}=0$ for $j \geq j+3$, i.e. $A$ is of the form $$A= \begin{bmatrix}
			* & \ldots &*&*&*&*\\
			* & \ldots &*&*&*&*\\
			1 & \ldots &*&*&*&*\\
			\vdots & \ddots & \vdots &\vdots &\vdots &\vdots\\
			0 & \ldots &1&*&*&*\\
			0 & \ldots &0&1&*&*
		\end{bmatrix}$$
		Then $A$ is similar to a matrix of the form $J^2 + \begin{bmatrix}
			0 & P\\
			0 & Q
		\end{bmatrix}$ for some $P \in M_{(n-2) \times 2}(\mathbb{Z}_p)$ and $Q \in M_2(\mathbb{Z}_p)$.
	\end{lem}
	\begin{lem} \label{Lemma 2.4 of Koo}
		Assume that $B \in M_n({\mathbb{Z}_p})$ satisfies $B_{i,i-1} =1$ for $3 \leq i \leq n$ and $B_{i,j} = 0$ for $1 \leq i \leq j \leq n-1$, i.e. $B$ is of the form $$B= \begin{bmatrix}
		0 & 0 & \ldots & 0 & 0 & *\\
		* & 0 & \ldots & 0 & 0 & *\\
		* & 1 & \ldots & 0 & 0 & *\\
		\vdots& \vdots & \ddots & \vdots & \vdots &\vdots \\
		* & * & \ldots & 1 & 0 & *\\
		* & * & \ldots & * & 1 & *\\
		\end{bmatrix}.$$ Then $B$ is similar to a matrix $A$ which satisfies conditions in Lemma \ref{Lemma 2.3 of Koo}.
	\end{lem}
	The proofs of these lemmas are exactly similar to the proof in ~\cite{Koo2025}. Now we prove the $\mathbb{Z}_p$-version of Lemma 2.5 of~\cite{Koo2025}, for odd prime $p$.
	\begin{lem} \label{Lemma 2.5 of Koo}
		Let $S \in M_n(\mathbb{Z}_p)$ and $a_1$ be unit in $\mathbb{Z}_p$. Then there is $X_1 \in M_n(\mathbb{Z}_p)$ such that $B=S-a_1X_1^2$ satisfies the conditions in Lemma \ref{Lemma 2.4 of Koo}. 
	\end{lem}
	\begin{proof}
		Let $T \in M_n(\mathbb{Z}_p)$ be upper-triangular matrix, $X=J+T$ and $x_{i,j}:= X_{i,j}=T_{i,j}$ for $1 \leq i \leq j \leq n$. We prove the lemma by choosing the $p$-adic integers $x_{i,j}$ which satisfy the equation
		\begin{equation} \label{Eq for triangular matrix}
			\begin{bmatrix}
				B_{1,i} \\
				\vdots\\
				B_{i,i}\\
				B_{i+1,i}
			\end{bmatrix} = \begin{bmatrix}
			0 \\
			\vdots\\
			0 \\
			1
			\end{bmatrix}
		\end{equation}
		for each $1 \leq i \leq n-1$.
		\begin{enumerate}[(a)]
			\item Choose $x_{1,1} = 0$. Then Equation \eqref{Eq for triangular matrix} for $i=1$ is equivalent to
			\begin{equation} \label{S_11 equation}
				\begin{bmatrix}
					S_{1,1} \\
					S_{2,1}
				\end{bmatrix} - a_1 \begin{bmatrix}
				x_{1,2}\\
				x_{2,2}
				\end{bmatrix} = \begin{bmatrix}
				0 \\
				1
				\end{bmatrix}.
			\end{equation}
			Since $a_1$ is unit, we can choose $x_{1,2}$ and $x_{2,2}$ satisfying Equation \eqref{S_11 equation}.
			\item Assume that we have chosen the integers $x_{i,2}$ ($1 \leq i$) which satisfy Equation \eqref{Eq for triangular matrix} for $1 \leq i \leq r-1$. Then Equation \eqref{Eq for triangular matrix} is equivalent to
			\begin{equation*}
				\begin{bmatrix}
					S_{1,r}\\
					\vdots\\
					S_{r+1,r}
				\end{bmatrix} - a_1 \begin{bmatrix}
				x_{1,r+1}\\
				\vdots \\
				x_{r+1,r+1}
				\end{bmatrix} = M
			\end{equation*}
			where $M \in M_{r+1,1}(\mathbb{Z}_p)$ is determined by $a_1$ and $x_{i,j}$ ($1 \leq i \leq j \leq r$). Since $a_1$ is unit, we can choose $x_{i,r+1}$ which satisfy the above equation.
		\end{enumerate}
	\end{proof}
	Now we state $p$-adic version of Lemma 2.6 of~\cite{Koo2025}, for an odd prime $p$.
	\begin{lem} \label{Lemma 2.6 for odd prime}
		Let $a_2$ and $a_3$ be units in $\mathbb{Z}_p$, $P \in M_{(n-2) \times 2}(\mathbb{Z}_p)$ and $Q \in M_2(\mathbb{Z}_p)$. Then there are $X_2$ and $X_3 \in M_n(\mathbb{Z}_p)$ such that $J^2 + \begin{bmatrix}
			0 & P\\
			0 & Q
		\end{bmatrix} - a_2X_2^2 = a_3 X_3^2 $. 
	\end{lem}
	\begin{proof}
		Since $a_2$ is unit in $\mathbb{Z}_p$, let $t_2 \in \mathbb{Z}_p$ such that $a_2t_2=1$. Let
		\begin{equation*}
			B_2 = s_1 E_{n-2,n} + s_2 E_{n-1,n} +s_3 E_{n,n}.
		\end{equation*}
		Let	$X_2 = J_{t_2} + B_2$. Hence
		\begin{eqnarray*}
		    X_2^2 &=& (J_{t_2} + B_2)^2 \\
			&=& J^2_{t_2} + J_{t_2} B_2 + B_2J_{t_2} +B_2^2 \\
			&=& t_2 J^2 + (J_{t_2} B_2 + B_2J_{t_2} +B_2^2)
		\end{eqnarray*}
		Now \begin{eqnarray*}
			J_2B_2 &=& s_1 E_{n-1,n} + t_2s_2 E_{n,n} \\
			B_2J_2 &=& (s_1 E_{n-2,n} + s_2E_{n-1,n} +s_3 E_{n,n}) (t_2 E_{n,n-1}) \\
			B_2^2 &=& s_1s_3 E_{n-2,n} + s_2s_3 E_{n-1,n} + s_3^2 E_{n,n}
		\end{eqnarray*}
		Since $a_2t_2=1$,
		\begin{equation} \label{Eq a2X2^2 for odd prime}
			a_2X_2^2 = J^2 +  \left[
			\begin{array}{c|cc}
				\textbf{0} & 0 & 0 \\ 
				\textbf{0} & s_1 & a_2s_1s_3 \\\hline
				\textbf{0} & s_2 & a_2 (s_1+s_2s_3) \\
				\textbf{0} & s_3 & s_2+ a_2s_3^2
			\end{array}
			\right]
		\end{equation}
		Let $L \in M_{(n-2) \times 2} (\mathbb{Z}_p)$ and $M = \begin{bmatrix}
			x & 1+x+x^2 \\
			-1 & -1-x
		\end{bmatrix}$
		Let $X_3 = \begin{bmatrix}
			0 & L \\
			0 & M
		\end{bmatrix}^2$. Note that $M^2 = \begin{bmatrix}
		-1-x & -1-x-x^2 \\
		1 & x
		\end{bmatrix}$ and $M^3 = I_2$. Therefore, $X_3 ^2 = \begin{bmatrix}
		0 & L \\
		0 & M
		\end{bmatrix}$. Hence
		\begin{equation} \label{Eq a3X3^2 for odd prime}
			a_3X_3^2 = \left[
			\begin{array}{c|cc}
				
				\textbf{0} & * &* \\ \hline
				\textbf{0} & a_3L_{n-2,1} & a_3L_{n-2,2} \\
				\textbf{0} & a_3x & a_3(1+x+x^2) \\
				\textbf{0} & -a_3 & a_3(-1-x)
			\end{array}
			\right]
		\end{equation}
		Let $Q = \begin{bmatrix}
			a & b \\
			c & d
		\end{bmatrix}$.
		From Equations \eqref{Eq a2X2^2 for odd prime} and \eqref{Eq a3X3^2 for odd prime},
		\begin{equation}
			J^2 + \begin{bmatrix}
				0 & P \\
				0 & Q
			\end{bmatrix} - a_2X_2^2 = a_3X_3^2
		\end{equation}
		if and only if the following equations are satisfied.
		\begin{subequations}
			\begin{align}
				a-s_2 &= a_3x  \label{Eq 10a} \\
				b- (a_2s_1 + a_2s_2s_3) &= a_3 (1+x+x^2) \label{Eq 10b}\\
				c-s_3 &= -a_3 \label{Eq 10c}\\
				d- (s_2 + a_2s_3^2) &= -a_3-a_3x \label{Eq 10d}\\
				P - \begin{bmatrix}
					0 & 0 \\
					s_1 & a_2s_1s_3
				\end{bmatrix} &= a_3L \label{Eq 10e}
			\end{align}
		\end{subequations}
		We can solve these equations for $x, s_1, s_2, s_3$ and $L$ as follows:
		\begin{enumerate}[(a)]
			\item From \eqref{Eq 10c}, find $s_3$.
			\item From \eqref{Eq 10a} and \eqref{Eq 10d}, $d-(s_2 + a_2s_3^2) = -a_3 -(a-s_2)$. Since $2$ is unit in $\mathbb{Z}_p$, we can find $s_2$. Since $a_3$ is unit, we can find $x$.
			\item Since $a_2$ is unit, we can find $s_1$ satisfying \eqref{Eq 10b}.
			\item Since $a_3$ is unit, we can find $L$ satisfying \eqref{Eq 10e}.
		\end{enumerate}
	\end{proof}
	The step (b) above is not applicable for $p=2$. We state the  $\mathbb{Z}_2$-version of Lemma 2.6 of ~\cite{Koo2025}. 
	\begin{lem} \label{Lemma 2.6 for p=2}
		Let $a_2,a_3,a_4$ be $2$-adic integers such that $a_2$ and $a_4$ are units, $P \in M_{(n-2) \times 2}(\mathbb{Z}_2)$ and $Q \in M_2(\mathbb{Z}_2)$. Then there are $X_2,X_3,X_4 \in M_n(\mathbb{Z}_2)$ such that $J^2 + \begin{bmatrix}
			0 & P\\
			0 & Q
		\end{bmatrix} - a_2X_2^2 - a_3 X_3^2 = a_4 X_4^2 $. 
	\end{lem}
	\begin{proof}
		Since $a_2$ is unit, $a_2$ and $a_3$ are coprime. Choose $t_2, t_3 \in \mathbb{Z}_2$ such that $a_2t_2 + a_3t_3 = 1$. Next we define two matrices $B_2$ and $B_3$ as follows:
		\begin{eqnarray*}
			B_2 &=& s_1 E_{n-2,n} + s_2 E_{n-1,n} +s_3 E_{n,n} \\
			B_3 &=& u_1 E_{n-2,n} + u_2 E_{n-1,n} + u_3 E_{n,n}
		\end{eqnarray*}
		We then define target matrices as $X_2 = J_{t_2} + B_2$ and $X_3= J_{t_3} + B_3$. We proceed as before.  		\begin{equation*}
			X_2^2 = t_2J^2 + 
			\left[
			\begin{array}{c|cc}
				\textbf{0} & 0 & 0 \\ \textbf{0} & t_2s_1 & s_1s_3 \\\hline
				\textbf{0} & t_2s_2 & s_1+s_2s_3 \\
				\textbf{0} & t_2s_3 & t_2s_2+s_3^2
			\end{array}
			\right]
		\end{equation*}
		Similarly, 
		\begin{equation*}
			X_3^2 = t_3J^2 + 
			\left[
			\begin{array}{c|cc}
				\textbf{0} & 0 & 0 \\ 
				\textbf{0} & t_3u_1 & u_1u_3 \\ \hline
				\textbf{0} & t_3u_2 & u_1+u_2u_3 \\ 
				\textbf{0} & t_3u_3 & t_3u_2+u_3^2
			\end{array}
			\right]
		\end{equation*}
	Since $a_2t_2 + a_3t_3 = 1$,
	\begin{equation} \label{Eq LHS for Z2}
		a_2 X_2^2 + a_3 X_3^2 = J^2 + \left[
		\begin{array}{c|cc}
			\textbf{0} & 0 & 0 \\ 
			\textbf{0} & d_1 & d_2 \\ \hline
			\textbf{0} & a_2t_2s_2 + a_3t_3u_2 & a_2(s_1+s_2s_3) + a_3(u_1+u_2u_3) \\
			\textbf{0} & a_2t_2s_3 + a_3 t_3u_3 & a_2(t_2s_2+s_3^2)+a_3(t_3u_2+u_3^2)
		\end{array}
		\right]
	\end{equation}
	for $d_1, d_2 \in \mathbb{Z}_2$. \\
	We will define $X_4$ as before. Let $L \in M_{(n-2) \times 2}(\mathbb{Z}_2)$ and  $M = \begin{bmatrix}
	x & 1+x+x^2 \\
	-1 & -1-x
	\end{bmatrix} \in M_2(\mathbb{Z}_2)$ and $X_4 = \begin{bmatrix}
	0 & L\\
	0 & M
	\end{bmatrix} ^2 \in M_n(\mathbb{Z}_2)$. Then,
	\begin{equation} \label{Eq RHS for Z2}
		a_4X_4^2 = \left[
		\begin{array}{c|cc}
			\textbf{0}& * & * \\ \hline
			\textbf{0} & a_4L_{n-1,2} & a_4L_{n-1,2} \\
			\textbf{0} & a_4x & a_4(1+x+x^2) \\
			\textbf{0} & -a_4 & a_4(-1-x)
		\end{array}
		\right]
	\end{equation}
	Let $Q = \begin{bmatrix}
		a & b \\
		c & d
	\end{bmatrix}$.
	From Equations \eqref{Eq LHS for Z2} and \eqref{Eq RHS for Z2},
	\begin{equation}
		J^2 + \begin{bmatrix}
			0 & P \\
			0 & Q
		\end{bmatrix} - a_2X_2^2 -a_3X_3^2= a_4X_4^2
	\end{equation}
	if and only if the following equations are satisfied.
	\begin{subequations}
		\begin{align}
			a-(a_2t_2s_2 + a_3t_3 u_2) &= a_4x  \label{Eq 14a} \\
			b- (a_2s_1+ a_3u_1)-(a_2s_2s_3 +a_3u_2u_3) &= a_4 (1+x+x^2) \label{Eq 14b}\\
			c-a_2t_2s_3 - a_3t_3u_3 &= -a_4 \label{Eq 14c}\\
			d- (a_2t_2s_2 + a_3t_3u_2) - a_2s_3^2 -a_3u_3^2&= -a_4-a_4x \label{Eq 14d}\\
			P - \begin{bmatrix}
				0 & 0 \\
				d_1 & d_2
			\end{bmatrix} &= a_4L \label{Eq 14e}
		\end{align}
	\end{subequations}
	We can solve these equations for $x, s_1, s_2, s_3, u_1,u_2,u_3$ and $L$ as follows:
	\begin{enumerate}[(a)]
		\item Since $a_2t_2 + a_3t_3 =1$, we can find $s_3$ and $u_3$ satisfing Equation \eqref{Eq 14c}.
		\item To solve \eqref{Eq 14d} and \eqref{Eq 14a}, we need $a_2s_3^2 + a_3u_3^2 \equiv a+d+a_4 (\mbox{mod 2})$.
		\item To achieve the above equivalence, we replace $s_3$ by $s_3 + a_3t_3$ and $u_3$ by $u_3-a_2t_2$, as in ~\cite{Koo2025}. This replacement doesn't change \eqref{Eq 14c}. Moreover, $a_2(s_3 + a_3t_3)^2 + a_3(u_3-a_2t_2)^2 \equiv a+d+a_4 (\mbox{mod 2})$.
		\item Let $a-d_3=d_4$ and and $d-d_3 - (a_2s_3^2 + a_3u_3^2) = -a_4-d_4$. Substitute $d_4$ in the the second equation, we get $2d_3 = a + d  +a_4 - (a_2s_3^2 + a_3u_3^2)$. We can now find $d_3$ satisfying this equation, and hence $d_4$ also where $d_3 = a_2t_2s_2 + a_3t_3u_2$ and $d+4=a_4x$. This gives us $s_2, u_2$ and $x$ satisfying Equations \eqref{Eq 14a} and \eqref{Eq 14d}.
		\item Since $a_2$ and $a_3$ are coprime, we can find $s_1,u_1$ satisfying Equation \eqref{Eq 14b}.
		\item Since $a_4$ is unit, we canfind $L$ satisfying Equation \eqref{Eq 14e}.
	\end{enumerate}
	\end{proof}
	Now we can find the upper bound for $f(n)$ for $n \geq 4$.
	\begin{thm}
		For every integer $n \geq 4$, $f(n) \leq 3$, if $p$ is odd prime. $f(n) \leq 4$, if $p=2$.
	\end{thm}
	\begin{proof}
		Suppose $p$ is an odd prime. Let $a_1,a_2$ and $a_3$ be units in $\mathbb{Z}_p$. We will prove that $\sum_{i=1}^{3} a_iX_i^2$ is universal over $M_n(\mathbb{Z}_p)$. Let $S \in M_n(\mathbb{Z}_p)$. By Lemmas \ref{Lemma 2.3 of Koo}, \ref{Lemma 2.4 of Koo} and \ref{Lemma 2.5 of Koo}, there exist $X_1 \in M_n(\mathbb{Z}_p)$ such that $B= S- a_1X_1^2$ is similar to the matrix of the form $J^2 + \begin{bmatrix}
			0 & P\\
			0 & Q
		\end{bmatrix}$ for some $P \in M_{(n-2) \times 2}(\mathbb{Z}_p)$ and $Q \in M_2(\mathbb{Z}_p)$. Let $U \in M_n(\mathbb{Z}_p)$ be an invertible matrix such that $UBU^{-1} = J^2 + \begin{bmatrix}
		0 & P\\
		0 & Q
		\end{bmatrix}$.\\
		
		By Lemma \ref{Lemma 2.6 for odd prime}, there are $X_2$ and $X_3$ such that $J^2 + \begin{bmatrix}
			0 & P \\
			0 & Q
		\end{bmatrix} = a_2X_2^2 + a_3X_3^2$.
		now
		\begin{equation*}
			S= a_1X_1^2 + U^{-1}(J^2 + \begin{bmatrix}
				0 & P\\
				0 & Q
			\end{bmatrix})U = a_1X_1^2 + a_2 \left(U^{-1}X_2 U\right)^2 + a_3 (U^{-1}X_3 U)^2.
		\end{equation*}
		Hence $f(n) \leq 3$. Similarly we can prove that $f(n) \leq 4$ for $p=2$.
	\end{proof}
	\begin{cor}
		Any matrix over $M_n(\mathbb{Z}_p)$ (for odd prime $p$), $n \geq 4 $ can be expressed as a sum of at most three squares. Any matrix over $M_n(\mathbb{Z}_2)$ can be expressed as a sum of at most four squares.
	\end{cor}
	\bibliography{library}
	\bibliographystyle{vancouver}
\end{document}